\documentclass[a4paper]{amsart}
\usepackage[latin9]{inputenc}
\usepackage{a4}
\usepackage{amsfonts}
\usepackage{amssymb}
\usepackage{amsmath}
\usepackage{amsthm}
\usepackage{changepage}
\usepackage{nomencl}
\usepackage{geometry}
\usepackage{tikz}
\usepackage{verbatim}
\usetikzlibrary{matrix}
\usepackage{hyperref}
\usepackage[toc,page]{appendix}
\usepackage{mathrsfs}
\usepackage[square, numbers, comma]{natbib}

\usepackage{enumerate}
\usepackage{amsbsy}

\usepackage[all]{xy}
\usepackage{pb-diagram,pb-xy}
   \bibpunct{[}{]}{,}{n}{}{,}
\input cyracc.def

\begin{document}
\providecommand{\keywords}[1]{\textbf{\textit{Keywords: }} #1}
\newtheorem{theorem}{Theorem}[section]
\newtheorem{lemma}[theorem]{Lemma}
\newtheorem{proposition}[theorem]{Proposition}
\newtheorem{corollary}[theorem]{Corollary}
\newtheorem{problem}[theorem]{Problem}
\newtheorem{question}[theorem]{Question}
\newtheorem{conjecture}[theorem]{Conjecture}
\newtheorem{claim}[theorem]{Claim}
\newtheorem{condition}[theorem]{Condition}

\theoremstyle{definition}
\newtheorem{definition}[theorem]{Definition} %[section]
\theoremstyle{remark}
\newtheorem{remark}[theorem]{Remark}
\newtheorem{example}[theorem]{Example}
\newtheorem{condenum}{Condition}

\def\p{\mathfrak{p}}
\def\q{\mathfrak{q}}
\def\s{\mathfrak{S}}
\def\Gal{\mathrm{Gal}}
\def\Ker{\mathrm{Ker}}
\def\soc{\mathrm{soc}}
\def\Coker{\mathrm{Coker}}
\newcommand{\cc}{{\mathbb{C}}}   
\newcommand{\ff}{{\mathbb{F}}}  
\newcommand{\nn}{{\mathbb{N}}}   
\newcommand{\qq}{{\mathbb{Q}}}  
\newcommand{\rr}{{\mathbb{R}}}   
\newcommand{\zz}{{\mathbb{Z}}}  
\def\K{\kappa}

\title{A note on monogenic even polynomials}
\author{Joachim K\"onig}
\address{Department of Mathematics Education, Korea National University of Education, Cheongju 28173, South Korea}
\email{jkoenig@knue.ac.kr}
\begin{abstract}
We extend several predecessor works on even sextic monogenic polynomials. In particular, we prove a conjecture of Lenny Jones \cite{Jones}, thereby classifying even sextic monogenic polynomials with cyclic Galois group. 
This result is key to completing previous partial results on existence or non-existence of infinite families of even sextic monogenic polynomials with a prescribed Galois group (\cite{Lavallee}).
Some of the underlying ideas are relevant for investigation of more general families of even polynomials $f(X^2)$, or power-compositional polynomials $f(X^\ell)$.
\end{abstract}
\maketitle

\section{Main results}

A number field $K$ is called {\it monogenic}, if 
 its ring of integers $\mathcal{O}_K$ is of the form $\mathbb{Z}[\alpha]$ for a suitable $\alpha\in \mathcal{O}_K$. In this case, $\alpha$ is called a {\it monogenerator} of $\mathcal{O}_K$. 
Let $f\in \mathbb{Z}[X]$ be an irreducible monic integer polynomial and $\alpha\in \overline{Q}$ a root of $f$. Call $f$ {\it monogenic}, if the ring $\mathbb{Z}[\alpha]$ is the full ring of integers of $K:=\mathbb{Q}(\alpha)$. It is well-known that this is equivalent to the polynomial discriminant $\Delta(f)$ equalling the field discriminant $\Delta(K)$; more generally a prime number $p$ divides the index $[\mathcal{O}_K:\mathbb{Z}[\alpha]]$ if and only if it divides the integer $\frac{\Delta(f)}{\Delta(K)}$.

In this note, we prove the following result.
\begin{theorem}
\label{thm:main}
Let $f(X)=X^6+aX^4+bX^2+c$ with $a,b,c\in \mathbb{Z}$ such that $f$ is irreducible and $\Gal(f/\mathbb{Q})\cong C_6$. Then $f$ is monogenic if and only if 
$$(a,b,c)\in \{(-7,14,-7), (-6,9,-3), (5,6,1), (6,5,1), (6,9,1), (9,6,1)\}.$$
\end{theorem}
Compare \cite[Theorem 1.1]{Jones}, which classified the monogenic polynomials inside certain (one- or two-parameter) subfamilies of the family of all even cyclic sextics, in particular discovering all the six monogenic cases above. That these are in fact the only monogenic cases inside the whole family was conjectured, based on computational evidence, in Remark 3.1 of the same paper. 
Note that every cyclic sextic number field, and more generally every sextic field with a cubic subfield, can be generated by an even integral polynomial as above, but it is not necessarily true that a monogenerator of a monogenic such number field must have such a minimal polynomial, so that the above theorem must not be confused with a full classification of cyclic sextic monogenic number fields (to give just one example, the field generated by $X^6+21X^4+35X^2+7$ is $\mathbb{Q}(\zeta_7)$ and hence monogenic, whereas the polynomial itself is not). Further investigations on the monogenicity of cyclic sextic {\it fields} (of certain particular shapes) have been undertaken in \cite{GR} (showing that all but finitely many of the so-called ``simplest sextic" fields are non-monogenic) and \cite{Gaal}. %, however uncovering no new fields other than the ones already generated by some of the polynomials in Theorem \ref{thm:main}. 
For a broader overview of results on the monogenicity of (cyclic or non-cyclic) sextic fields, see also \cite[Chapter 11]{Gaal_Book}.

For comparison, we also give an analog of Theorem \ref{thm:main} for certain higher degrees. In spite of the larger degrees, this will be more straightforward to deduce than the $C_6$ case, essentially due to the paucity of cyclic monogenic fields of prime degree $q\ge 5$ (cf.\ \cite{Gras}).
\begin{lemma}
\label{lem:c10andmore}
For any prime $q\ge 5$ there are at most finitely many monogenic even polynomials $f(X)=g(X^2)\in \mathbb{Z}[X]$ of degree $2q$ with $\Gal(f/\mathbb{Q})\cong C_{2q}$. In the range $5\le q\le 19$, the only cases are
{\footnotesize{
\begin{itemize}
\item[1)] $q=5$, and $f\in \{X^{10} + 9X^8 + 28X^6 + 35X^4 + 15X^2 + 1, X^{10} - 11X^8 + 44X^6 - 77X^4 + 55X^2 - 11,  X^{10} + 15X^8 + 35X^6 + 28X^4 + 9X^2 + 1\}$.
\item[2)] $q=11$, and  $f\in \{X^{22} + 21X^{20} + 190X^{18} + 969X^{16} + 3060X^{14} + 6188X^{12} + 8008X^{10} + 6435X^8 + 3003X^6 + 715X^4 + 66X^2 + 1, 
    X^{22} - 23X^{20} + 230X^{18} - 1311X^{16} + 4692X^{14} - 10948X^{12} + 16744X^{10} - 16445X^8 + 9867X^6 - 3289X^4 + 506X^2 - 23,  
    X^{22} + 66X^{20} + 715X^{18} + 3003X^{16} + 6435X^{14} + 8008X^{12} + 6188X^{10} + 3060X^8 + 969X^6 + 190X^4 + 21X^2 + 1\}$.
\end{itemize}}}
\end{lemma}
%The hits seem to be $K(\sqrt{-1})$ and $K(\sqrt{2q+1})$, where $K=\mathbb{Q}(\zeta_{2q+1})\cap \mathbb{R}$; analogous fact still true for the degree-6 C_6 fields. 
% [i.e., exactly the two (out of three) index-2 subfields of \mathbb{Q}(\zeta_{4(2q+1)})] ramified at 2.]
%Actually clear a priori that these are the only candidates, since (by lemma) no newly ramified prime other than 2 is allowed in the top quadratic extension (and conversely, 2 needs to ramify).

In Section \ref{sec:general}, we will use Theorem \ref{thm:main} (and a few more elementary additional observations) to give a comprehensive overview of even monogenic sextics with prescribed Galois groups (Theorem \ref{thm:general}), and in particular show that among all the possible Galois groups of even sextics, $C_6$ and $S_3$ are the only ones not possessing an infinite family of monogenic even sextics. A comparable classification for even {\it octics} (although with certain additional restrictions) was recently given in \cite{Jones2}. 
Note also that the existence of infinitely many monogenic sextic {\it fields} with Galois group $S_3$ is known (e.g., \cite[Theorem 3.8]{GK}), whereas for the group $C_6$, this seems to be an open problem.

%Instead, newly discovered monogenic field, not listed as known in LMFDB: x^6 - 12*x^4 + 20*x^2 - 8 (also gen. by e.g. x^6 - 10*x^4 + 24*x^2 - 8, see LMFDB; generated by monogenic polynomial x^6 - 12*x^5 + 34*x^4 - 22*x^3 - 21*x^2 + 14*x + 7 )
%At least one more in Gaal's book

\iffalse
%Full classification for A_4 case also in reach?
\marginpar{Warning: ``Only if" of Theorem \ref{thm:a4} is so far not proven! Some ideas, but computationally tricky!}
\begin{theorem}
\label{thm:a4}
Let $f(X)=X^6+aX^4+bX^2+c$ with $a,b,c\in \mathbb{Z}$ such that $f$ is irreducible and $\Gal(f/\mathbb{Q})\cong A_4$. Then $f$ is monogenic if and only if one of the following holds:
\begin{itemize}
\item[a)] $c=-1$, $b=a-3$; and $a^2 - 3a + 9$ is either a squarefree number or three times a squarefree number.
\item[b)] $c=-1$ and moreover $(a,b)\in \{(-20,-9), (-3,-4), (4,3), (9,20)\}$.
\end{itemize}
\end{theorem}
\fi

\section{Some preparatory lemmas}
We need a couple of basic lemmas. The first two are fairly well-known, but we include proofs for completeness.
\begin{lemma}
\label{lem:monogen}
Let $g\in \mathbb{Z}[X]$ be monic and $d\in \mathbb{N}$ such that $f(X):=g(X^d)$ is irreducible and monogenic. Then $g$ is irreducible and monogenic as well.
\end{lemma}
\begin{proof} Let $\alpha$ be a root of $f$. Since $\alpha^d$ is a root of $g$ and $\mathbb{Q}(\alpha)$ is (at most) a degree-$d$ extension of $\mathbb{Q}(\alpha^d)$, irreducibility of $f$ implies irreducibility of $g$. Furthermore, by assumption every algebraic integer in $\mathbb{Q}(\alpha)$ can be written uniquely as an integral linear combination of $\alpha^0,\dots, \alpha^{rd-1}$, where $r:=\deg(g)$. At the same time, every element of $\mathbb{Q}(\alpha^d)$ can be written uniquely as a rational linear combination of $\alpha^0,\alpha^d, \dots \alpha^{(r-1)d}$. It follows that every algebraic integer of $\mathbb{Q}(\alpha^d)$ is a unique integral linear combination of $\alpha^0,\alpha^d,\dots, \alpha^{(r-1)d}$, showing that $\mathbb{Z}[\alpha^d]$ is the full ring of integers of $\mathbb{Q}(\alpha^d)$.
\end{proof}

\begin{lemma}
\label{lem:eisen}
Let $K$ be a number field with monogenerator $\alpha$, and assume that there exists a prime $p$ which is totally ramified in $K$ and such that $\alpha$ is of norm divisible by $p$. Then the minimal polynomial of $\alpha$ over $\mathbb{Q}$ is $p$-Eisenstein.
\end{lemma}
\begin{proof}
Let $\mathfrak{p}$ be the unique prime of $K$ extending $p$. Since $\alpha$ is of norm $N(\alpha)$ divisible by $p$, it must be of positive $\mathfrak{p}$-adic valuation, and since $\alpha$ is a monogenerator, it must be of $\mathfrak{p}$-adic valuation $1$ (or else, an integral element of this valuation could not be an integral linear combination of powers of $\alpha$). Hence, $N(\alpha)$ is strictly divisible by $p$. Moreover, since $p$ is totally ramified in $K/\mathbb{Q}$, the minimal polynomial $f$ of $\alpha$ must reduce modulo $p$ to a power of a linear polynomial. However, $f\mod p$ is divisible by $X$ by assumption, so that $f\mod p$ must be a monomial. In total, we have obtained that $f$ is $p$-Eisenstein.
\end{proof}

The following lemma can essentially be read out of \cite[Lemma 2.3]{Jones}, which states that the reducibility of a certain resolvent polynomial implies the below shape for an even sextic polynomial (the converse being verifiable by direct calculation), 
without explicitly naming the associated Galois group $D_6$.
\begin{lemma}
\label{lem:d6}
Let $f(X)=X^6+aX^4+bX^2+c\in \mathbb{Z}[X]$ be irreducible. 
%Maybe assuming separability is enough?
Then $\Gal(f/\mathbb{Q})$ is a subgroup of $D_6$ if and only if there exist integers $m,n\in \mathbb{Z}$ such that ($c$ divides $n^2$ and)
$$a=n^2/c-2m, \  b=m^2-2n.$$
\end{lemma}

\begin{lemma}
\label{lem:c61}
Let $m,n\in \mathbb{Z}$, $c|n^2$ and let $f(X)=X^6+(n^2/c-2m)X^4+(m^2-2n)X^2+c$. Then $\Gal(f/\mathbb{Q})\cong C_6$ if and only if all of the following hold:
\begin{itemize}
\item[i)] $-c$ is not a square.
\item[ii)] $g(X):=f(\sqrt{X})$ is irreducible.
\item[iii)]
$d(m,n,c):= -(4m^3c - m^2n^2 - 18mnc + 4n^3 + 27c^2)\in \mathbb{Z}$ is a square.
\end{itemize}
\end{lemma}
\begin{proof}
Assume first that $\Gal(f/\mathbb{Q})\cong C_6$. The discriminant of $f$ equals $-c$ up to a square factor. Since $C_6$ is not contained in the alternating group $A_6$, a $C_6$-extension of $\mathbb{Q}$ cannot have square discriminant, whence i) is necessary. Furthermore, 
the cubic subextension generated by a root of $g(X)$ needs to be Galois of group $C_3$. This means that ii) and iii) are necessary, the latter since the discriminant of $g$ equals $d(m,n,c)$ up to squares.
Assume now conversely that i), ii) and iii) hold. Then the splitting field of $f$ contains the quadratic extension $\mathbb{Q}(\sqrt{-c})/\mathbb{Q}$ and a $C_3$-subextension (namely, the splitting field of $g$), i.e., contains a $C_6$ subextension. 
From Lemma \ref{lem:d6}, we know that $\Gal(f/\mathbb{Q})\le D_6$, and since the only subgroup of $D_6$ possessing a quotient $C_6$ is $C_6$ itself, it follows that  $\Gal(f/\mathbb{Q})= C_6$.
\end{proof}

\section{Proof of the main results}

%Prelim. lemmas

\begin{lemma}
\label{lem:evenmono}
Let $\ell<q$ be prime numbers, let $g(X)=X^{q}+ a_{q-1}X^{q-1}+\dots + a_1X + a_0\in \mathbb{Z}[X]$, and assume that $f(X):=g(X^\ell)$ is irreducible and monogenic, and such that $\Gal(f/\mathbb{Q})$ a subgroup of the wreath product $S_{\ell} \wr C_q (=S_\ell^q\rtimes C_q)$ not containing $C_\ell\wr C_q$. Then $f$ is $p$-Eisenstein for every prime $p|a_0$. I.e., $a_0$ is a squarefree number dividing all other coefficients $a_i$. Moreover, every prime divisor of $a_0$ is either $q$ or congruent to $1$ modulo $q$.
\end{lemma}

\begin{remark}
\label{rem:evenmono}
 Our main focus for application of Lemma \ref{lem:evenmono} will be on the case $\ell=2, q=3$. In fact, for primes $q\ge 5$, it is known by \cite{Gras} that the only monogenic $C_q$ number fields (existence of which is necessary by Lemma \ref{lem:monogen} to obtain a monogenic $f$ as above) are the maximal real subfields of cyclotomic fields $\mathbb{Q}(\zeta_{2q+1})$, where $2q+1$ must additionally be prime. Nevertheless, classification of monogenic polynomials $f=g(X^\ell)$ as above for $\deg(g)=q\ge 5$ may be of interest, see Lemma \ref{lem:c10andmore}.
 A useful observation for consideration of certain {\it non-cyclic} groups is the well-known fact that the inertia group at $p$ in the splitting field of a $p$-Eisenstein polynomial must be a transitive subgroup of the Galois group, and cyclic as soon as $p$ is tamely ramified. In the only possible case $p=q$ of wild ramification allowed by the assertion of Lemma \ref{lem:evenmono}, the inertia group would nevertheless have to be a subgroup of $S_{\ell}\wr C_q$ 
containing (the wild inertia group) $C_q$ as a normal subgroup and therefore be contained in $S_{\ell}\times C_q$. Transitivity then still implies the existence of a subgroup isomorphic to $C_{\ell}\times C_{q}\cong C_{\ell q}$. This means that for those groups $G$ as in Lemma \ref{lem:evenmono} which {\it do not} possess any elements of order $\ell q$ (e.g., the group $A_4$ in its transitive action on six points), there cannot be any prime $p$ as in the assertion of Lemma \ref{lem:evenmono}, i.e., one necessarily has $a_0=\pm 1$ for all monogenic  $f(X)=g(X^\ell)$ with Galois group $G$.  
\end{remark}

\begin{proof}[Proof of Lemma \ref{lem:evenmono}]
Let $\alpha$ be a root of $f$, let $p$ be a prime divisor of $a_0$, and set $g(X):=f(\sqrt[\ell]{X})$. First note that the second assertion follows from the first by considering the splitting field of $g$. By assumption, this must be a $C_q$-extension of $\mathbb{Q}$, and by the first assertion, its completion at $p$ is a totally ramified $C_q$-extension of $\mathbb{Q}_p$. If $p\ne q$, then this completion is thus totally tamely ramified, and it is well-known that such an extension is generated over $\mathbb{Q}_p$ by the $q$-th root of a suitable $p$-adic number. But then, it can only be cyclic if $\mathbb{Q}_p$ contains the $q$-th roots of unity, implying $p\equiv 1 \bmod q$.

To prove the first assertion, it suffices by Lemma \ref{lem:eisen} to show that $p$ is totally ramified in $\mathbb{Q}(\alpha^\ell)$ (since then $\alpha^\ell$ is of norm divisible by $p$ and therefore $g$, and thus $f$ is $p$-Eisenstein). 
Assume first that $p\le\ell$. Then the minimal polynomial $g$ of $\alpha^\ell$ has the root $0$ modulo $p$ and must thus be inseparable modulo $p$; indeed, a separable polynomial in $\mathbb{F}_p[X]$ of degree $q>p$ cannot split into linear factors, and thus the Frobenius at $p$ in $\mathbb{Q}(\alpha^\ell)$ would be a non-identity element fixing at least one point, something impossible in a $C_q$-extension. Now inseparability of $g(X)\bmod p$ together with the monogenicity of $g$ guaranteed by Lemma \ref{lem:monogen} implies that $p$ is (totally) ramified in $\mathbb{Q}(\alpha^\ell)/\mathbb{Q}$, completing the proof of the case $p\le\ell$.\footnote{Of course, the second assertion now implies that this case does not in fact occur.} 
So assume from now on $p> \ell$.

By assumption, $f$ is inseparable modulo $p$, with at least an $\ell$-fold root at $X=0$. If $g$ were also inseparable, then monogenicity of $f$ together with Lemma \ref{lem:monogen} would imply that $p$ is ramified in $\mathbb{Q}(\alpha^\ell)/\mathbb{Q}$, and we are done. So assume that $g$ is separable. But then $f(X)=g(X^\ell)$ can have at most (and in fact exactly, since $f$ is inseparable) one multiple root in $\overline{\mathbb{F}_p}$ (namely, at $X=0$, of multiplicity exactly $\ell$).\footnote{Here we have used $p\ne \ell$, since the map $X\mapsto X^\ell$ is then $\ell$-to-$1$ outside $X=0$.} Since $p>\ell$, the prime $p$ is tamely ramified in the splitting field of $f$ of discriminant exponent at most $\ell-1$, with equality if and only if the inertia group at $p$ is generated by an $\ell$-cycle. Monogenicity of $f$ thus implies that the latter must be the case. It is however an elementary exercise in permutation group theory to show that a transitive subgroup of $S_\ell\wr C_q$ containing one $\ell$-cycle must contain a full subgroup $C_\ell^q$ (by transitive action of $C_q$ on the $q$ copies of $S_\ell$), contradicting our assumption.
\end{proof}

The following lemma, using again the shape for $f(X)$ identified by Lemma \ref{lem:d6}, contains the crucial reduction of the family of $C_6$ polynomials in Theorem \ref{thm:main} to a very restricted subfamily.

\begin{lemma}
\label{lem:c6mono}
Let $m,n\in \mathbb{Z}$ and $c|n^2$ such that $f(X)=X^6+(n^2/c-2m)X^4+(m^2-2n)X^2+c$ is irreducible and monogenic, and has Galois group $C_6$. Then 
either  $\{m,n\} = \{-1,-2\}$ and $c=1$, or $mn=0$ and $c\equiv 1\bmod 4$. 
\end{lemma}
\begin{proof} 
Due to Lemma \ref{lem:evenmono}, we may and will assume that $c$ is odd, squarefree and divides the other coefficients. Since $c|n^2$ was already assumed, this implies $c|n$ and $c|m$, i.e., $m=jc$ and $n=kc$ for suitable $j,k\in \mathbb{Z}$. We then have \begin{equation}
\label{eq:delta}
\Delta(f)=-64 c^5 (jkc-1)^4  (\tilde{d}(j,k,c))^2,
\end{equation}
 where $\tilde{d}(j,k,c):=4j^3c^2 - j^2k^2c^2 - 18jkc + 4k^3c + 27$. 
 It follows readily that $c\equiv 1\bmod 4$, since $2$ must be wildly ramified and hence ramify in the quadratic subfield $\mathbb{Q}(\sqrt{\Delta(f)})=\mathbb{Q}(\sqrt{-c})$. Moreover, the discriminant of the cubic $g(X)=f(\sqrt{X})$ equals $-\tilde{d}(j,k,c)$ up to square factors. As in Lemma \ref{lem:c61}, this enforces $-\tilde{d}(j,k,c)$ to be a square. Now let $p$ be a prime divisor of $jkc-1$, i.e., $j$, $k$ and $c$ are all coprime to $p$ and $c\equiv \frac{1}{jk}\bmod p$. Plugging this into the cubic $g(X)$, we obtain $g(X)\equiv (X-\frac{1}{k})^2(X+\frac{k}{j})\bmod p$. Monogenicity implies that $p$ ramifies in the splitting field of $g$, and since this is a $C_3$-extension of $\mathbb{Q}$, $p$ must be totally ramified. Thus, necessarily $g(X)\equiv (X-\frac{1}{k})^3 \bmod p$, i.e., $j\equiv -k^2\bmod p$. Plugging this, and the condition $c\equiv -\frac{1}{k^3}\bmod p$ thus obtained, into the definition of $\tilde{d}(j,k,c)$, we obtain $\tilde{d}\equiv 0 \bmod p$, and thus even $\tilde{d}\equiv 0\bmod p^2$, since $-\tilde{d}$ is a square. Therefore, from \eqref{eq:delta}, $p^8 | \Delta(f)$, and if $p=2$, then even $p^{14}| \Delta(f)$. But the maximal discriminant exponent for a tamely (resp., wildly) ramified prime in the field discriminant of a $C_6$-extension is $5$ (resp., $9$). Due to the monogenicity assumption, this leaves only the option $p=3$. For this, we can verify ad hoc (going through all possibilities modulo $9$) that with our restrictions on $j$, $k$ and $c$, the condition $9|\tilde{d}(j,k,c)$ implies $9|jkc-1$, whence in \eqref{eq:delta}, we actually get $3^{2\cdot 4 + 4}=3^{12}|\Delta(f)$, again contradicting monogenicity. In total, we have obtained that $jkc-1$ cannot have any prime divisors, i.e., $jkc=0$ or $jkc=2$. The first option corresponds to $mn=0$ (since obviously $c\ne 0$), whereas the latter one, due to the condition that $c$ is odd, corresponds to $c=\pm 1$ and $jk=\pm 2$. Since $c=-1$ is impossible by Lemma \ref{lem:c61}, it follows that $c=1$ and $\{m,n\} = \{-1,-2\}$ or $\{1,2\}$. But the last case would lead to $-\tilde{d}(j,k,c)=-23$, which is not a square. This completes the proof.
\end{proof}

\begin{proof}[Proof of Theorem \ref{thm:main}]
Lemma \ref{lem:c6mono} has reduced the consideration to $mn=0$ (and $c\equiv 1\bmod 4$), or $c=1$ and $\{m,n\}=\{-1,-2\}$ for the shape of $f$ as in that lemma. The latter case leads to the pair of reciprocal polynomials $f(X)=X^6+5X^4+6X^2+1$ and $f(X)=X^6+6X^4+5X^2+1$, both of which are easily verified to be monogenic.\footnote{Note that the case $c=1$ has also been dealt with in the proof of \cite[Lemma 2.4]{Lavallee}.}  
The two cases $n=0$ and $m=0$ correspond exactly to Theorem 1.1(2) and Theorem 1.1(3) of \cite{Jones}, respectively, so that we could end our proof here by citing these results. Alternatively, a short argument for the latter two cases goes as follows. 
Assume first $m=0$ and $n\ne 0$. By Lemma \ref{lem:evenmono}, $n=kc$ for some integer $k$, so that $f(X)=X^6+k^2c X^4 - 2kc X^2 + c$, and $\Delta(f)=-64 c^5 (4k^3 c+27)^2$.  Here $-4k^3c-27$ needs to be a square, since this equals, up to squares, the discriminant of the cubic polynomial $g(X)=f(\sqrt{X})$.  I.e., $c=-\frac{y^2+27}{4k^3}$ for some integer $y$. Plugging this into the defining equation of $g$ and reducing modulo any prime divisor $p$ of $y$, we get $g(X)\equiv (X-\frac{3}{k})^2(X-\frac{3}{4k})\bmod p$. As before, the fact that $g$ is a monogenic $C_3$-polynomial implies that this factorization must actually be the cube of a linear, and thus necessarily $p=3$ (note that the case $p|k$ does not require separate attention, since then $p|y^2+4k^3c=-27$). So $y^2=-4k^3c-27$ is a power of $9$. But once again, the maximal $3$-power dividing the field discriminant of a  $C_6$-extension is $3^9$. Thus, $-4k^3c - 27 \in \{1,9,81\}$. Since $c$ is additionally squarefree (by Lemma \ref{lem:evenmono}) and congruent to $1$ modulo $4$ (by Lemma \ref{lem:c6mono}), this leaves only the possibilities $(k,c) = (-3,1)$ and $(k,c)=(1,-7)$, leading to $f(X)=X^6 + 9X^4 + 6X^2 + 1$ and $f(X)=X^6-7X^4+14X^2-7$ respectively. We may verify directly that those are indeed both monogenic.

Analogously, in the case $n=0$ and $m\ne 0$, we know that $m=jc$ for some integer $j$, and thus get $f(X)=X^6-2jcX^4+j^2c^2X^2+c$, and $\Delta(f)=-64c^5(4j^3c^2+27)^2$, where $-4j^3c^2-27 = -4jm^2-27$ is necessarily a square, i.e., $j=-\frac{y^2+27}{4m^2}$ for some integer $y$. Plugging this (as well as $c=\frac{m}{j}$) in the defining equation of $g$ and reducing modulo any prime divisor $p$ of $y$, we get $g(X)\equiv (X-\frac{m}{3})^2(X-\frac{4m}{3}) \bmod p$, implying that $p=3$ (note that the case $p|m$ does note require separate attention, since then $p|y^2+4jm^2=-27$). Thus, $-4j^3c^2-27\in \{1,9,81\}$ in analogy with the previous case. Since $c$ is squarefree and $c\equiv 1\bmod 4$, this leaves only the possibilities $(j,c) = (-3,1)$ and $(j,c) = (-1,-3)$. These yield $f(X)=X^6 + 6X^4 + 9X^2 + 1$ and $f(X)=X^6-6X^4+9X^2-3$,  both of which are indeed monogenic.
\end{proof}

%We come to the proof of Lemma \ref{lem:c10andmore}.
\begin{proof}[Proof of Lemma \ref{lem:c10andmore}]
By Lemma \ref{lem:monogen}, $g(X):=f(\sqrt{X})$ needs to be a monogenic cyclic polynomial of degree $q$, which by \cite{Gras} enforces that ($2q+1$ is prime and) the splitting field of $g$ is $K=\mathbb{Q}(\zeta_{2q+1})\cap \mathbb{R}$.
For a fixed number field $K$, it is moreover well-known (\cite{Gyory}) that there are at most finitely many monogenerators up to equivalence, i.e., there exists a finite set $S\subset \mathcal{O}_K$ such that every monogenerator of $\mathcal{O}_K$ is of the form $\pm \alpha + a$ for some $\alpha\in S$ and $a\in \mathbb{Z}$. This translates to saying that there is a finite set $\mathcal{M}\subset \mathbb{Z}[X]$ of irreducible degree-$q$ polynomials $h\in \mathbb{Z}[X]$ such that every $g\in \mathbb{Z}[X]$ with $g(X^2)$ cyclic and monogenic must be of the form $g(X)=h(\pm X-a)$ for some $h\in \mathcal{M}$ and $a\in \mathbb{Z}$. On the other hand, Lemma \ref{lem:evenmono} together with the fact that $2q+1$ is the only ramified prime in $K$ implies that, in order for $f$ to be monogenic with Galois group $C_{2q}$, the constant coefficient of $g(X)=h(\pm X-a)$ must be $\pm 1$ or $\pm(2q+1)$, which is certainly possible for at most finitely many $a\in \mathbb{Z}$. This shows the finiteness assertion. To determine the explicit set of polynomials in the range $5\le q\le 19$, note that only $q=5$ and $q=11$ are relevant due to the requirement of $2q+1$ being prime. For these, the full set of monogenerators (up to equivalence) of $\mathcal{O}_K$ is known explicitly by \cite{MR}. From this, identifying the relevant $f(X)=g(X^2)$ merely amounts to a short computation.
\end{proof}

\begin{remark}
We end this section by remarking that, up to discarding reciprocals of already identified polynomials, all monogenic $f(X)=g(X^2)$ identified in Theorem \ref{thm:main} and Lemma \ref{lem:c10andmore} are such that $g$ is the minimal polynomial of $-\zeta_d-\zeta_d^{-1}\pm 2$ for $d\in \mathbb{N}$ with $\varphi(d) = 2\deg(g)$ (i.e., $d\in \{7,9\}$ for $\deg(g)=3$, and $d=2q+1$ for $\deg(g)=q\ge 5$ a Sophie Germain prime). Extending this to all primes $q\ge 5$ should be possible at least depending on existing conjectures on the set of all monogenerators of real cyclotomic fields (see \cite{MR}). Whether there was any reason to expect the analogous result a priori for $\deg(g)=3$, I do not know.
\end{remark}

\section{Remarks on even sextics with other Galois groups}
\label{sec:general}
As mentioned in \cite{Lavallee}, there are a total of eight permutation groups occurring as the Galois group of some irreducible even sextic polynomial over $\mathbb{Q}$. Of those, five were identified in \cite{Lavallee} to occur infinitely often as the Galois group of a {\it monogenic} even sextic (even with the additional restriction that the constant coefficient should be $\pm 1$), the remaining three being $C_6$, $S_3$ and $S_4$ acting as $6T8$ (with the transitive group label as used, e.g., in GAP or Magma). Theorem \ref{thm:main} completely settles the case $C_6$. The remaining two groups can be handled without much extra effort, thereby yielding the following:
\begin{theorem}
\label{thm:general}
Let $G<S_6$ be a group occurring as the Galois group of an irreducible even sextic $f=X^6+aX^4+bX^2+c\in \mathbb{Z}[X]$. Then one of the following holds:
\begin{itemize}
\item[a)] $G=S_4\times C_2$, $A_4\times C_2$, $A_4$, $D_6$, or $S_4$ in its action as either $6T7$ or $6T8$; and in each case there are infinitely many monogenic even sextics with Galois group $G$.
\item[b)] $G=C_6$, and there are a total of six monogenic even sextics with Galois group $G$.
\item[c)] $G=S_3$, and there are no monogenic even sextics with Galois group $G$.
\end{itemize}
\end{theorem}
\begin{proof}
Due to \cite{Lavallee} and Theorem \ref{thm:main}, it remains to deal with case c), as well as the case $G= 6T8$ of case a). 

For the latter, one may consider the family of polynomials $X^6+9X^4+bX^2+b\in \mathbb{Z}[X]$. The discriminant of $g(X):=f(\sqrt{X})$ is of the form ``$(-b)$ times a square", whence it follows from \cite[Proposition 2.1]{AJ} that $\Gal(f/\mathbb{Q})$ embeds into $6T8$. Direct calculation at isolated values for $b$ (e.g., $b=1$) shows that, generically, $\Gal(f/\mathbb{Q})=6T8$, and hence by Hilbert's irreducibility theorem, this is the Galois group for all but a density-$0$ subset of all integers $b$. Using the discriminants $\Delta(g)=-4b(b-27)^2$ and $\Delta(f)=-2^{10}b^3(b-27)^4$ (as well as, e.g., the local behavior at the prime $2$), it is now an easy exercise to verify that $f$ is monogenic for all $b\equiv 2\bmod 4$ (which preserve the Galois group and) such that $b(b-27)$ is squarefree. That there are infinitely many $b$ with this property is easy to verify (e.g., if additionally $b$ is chosen coprime to $3$, then the squarefreeness condition just amounts to simultaneous squarefreeness of $b$ and $b-27$).

Finally, the case $G=S_3$ is contained in the more general assertion of Lemma \ref{lem:aux} below, thereby ending the proof.
\end{proof}

\begin{lemma}
\label{lem:aux}
There are no monogenic even polynomials $f(X)=g(X^2)$ with $g$ of odd degree $m\ge 3$ such that $Gal(f/\mathbb{Q})\cong D_m$ is the dihedral group (in its regular action).
\end{lemma}
\begin{proof}
Assume on the contrary that such $f,g$ exist. Firstly, as a special case of \cite[Theorem 2.7]{HJ}, $\Delta(f)$ is divisible by $\Delta(g)^2\cdot 2^{2m}$. Since $f$ is monogenic and has the same Galois closure as $g$, the prime $2$ must ramify in the field generated by a root of $g$, and hence (since $g$ is monogenic as well by Lemma \ref{lem:monogen}) $2$ must divide $\Delta(g)$. More concretely, $2$ must be wildly ramified, since the maximal possible exponent of a tamely ramified prime in the field discriminant of a degree-$2m$ extension would be $2m-1$. But the smallest possible discriminant exponent for the prime $2$ in a wildly ramified degree-$m$ extension with Galois group (of the Galois closure) $D_m$ is $m-1$ (indeed, from the fact that the inertia and decomposition group at $2$ must then both be generated by an involution with one fixed point, it follows that there are a total of $\frac{m-1}{2}$ primes extending $2$, each contributing exponent at least $2$ to the discriminant). %
Therefore $2^{m-1}| \Delta(g)$, and consequently $2^{2(m-1)+2m}=2^{4m-2}| \Delta(f)$. But the maximal possible exponent of $2$ in the discriminant of a degree-$2m$ extension with Galois group $D_{2m}$ can be easily calculated, as follows: firstly, since the wild inertia subgroup is a $2$-group normal in the full inertia group and the order-$2$ subgroups of $D_m$ are self-normalizing (since $m$ is odd), the inertia group must be cyclic and generated by an involution. Next, the maximal possible discriminant exponent of $2$ in a quadratic extension of $\mathbb{Q}$ is $3$, and due to the follow-up degree-$m$ extension being unramified, the maximal possible exponent in the full degree-$2m$ extension is $3m$. But $3m<4m-2$, showing that $f$ cannot be monogenic.
\end{proof}

\end{document}